\documentclass[oneside,english]{amsart}
\usepackage[T1]{fontenc}
\usepackage[latin9]{inputenc}
\usepackage{amsthm}
\usepackage{amstext}
\usepackage{amssymb}
\usepackage{esint}
\usepackage{graphicx}
\usepackage{enumerate}

\usepackage{graphicx}
\graphicspath{{noiseimages/}}

\makeatletter
\newtheorem{theorem}{Theorem}[section]
\newtheorem{lemma}[theorem]{Lemma}

\newtheorem{corollary}[theorem]{Corollary}
\numberwithin{figure}{section}
\theoremstyle{definition}

\newtheorem{example}[theorem]{Example}

\theoremstyle{remark}
\newtheorem{remark}[theorem]{Remark}

\numberwithin{equation}{section}
\makeatother

\usepackage{babel}

	\DeclareMathOperator{\dist}{dist}
	\DeclareMathOperator{\loc}{loc}
	
	\DeclareMathOperator*{\esssup}{ess\,sup}
	
	\DeclareMathOperator{\diam}{diam}
	
\begin{document}

\title[Spectral Properties]{On the First Eigenvalue of the Degenerate $p$-Laplace Operator in Non-Convex Domains}

\author{V.~Gol'dshtein, V.~Pchelintsev, A.~Ukhlov}

\begin{abstract}
In this paper we obtain lower estimates of the first non-trivial  eigenvalues of the 
degenerate $p$-Laplace operator, $p>2$, in a large class of non-convex domains. 
This study is based on applications of the geometric theory of composition 
operators on Sobolev spaces that permits us to estimates constants of  
Poincar\'e-Sobolev inequalities and as an application to  derive lower 
estimates of the first non-trivial  eigenvalues for the Alhfors domains
(i.e. to quasidiscs). This class of  domains includes some snowflakes type domains with fractal boundaries.
\end{abstract}

\maketitle

\section{Introduction}

In this article we consider the Neumann eigenvalue problem for the two-dimensio\-nal degenerate $p$-Laplace operator ($p>2$)
\[
\Delta_p u=\textrm{div}(|\nabla u|^{p-2}\nabla u).
\]

This operator arises in study of vibrations of nonelastic membranes \cite{Pol15}. The weak statement of the frequencies problem  
for the vibrations of a nonelastic membrane is equivalent to the follows spectral problem: to find eigenvalues $\mu_p$ and  eigenfunctions
$u \in W^{1}_{p}(\Omega)$ for the following variational problem
\begin{multline*}
\iint\limits _{\Omega} (|\nabla u(x,y)|^{p-2}\nabla u(x,y) \cdot \nabla v(x,y))\,dxdy \\ 
{} = \mu_p \iint\limits _{\Omega} |u(x,y)|^{p-2} u(x,y) v(x,y)\,dxdy,\,\,p>2, 
\end{multline*} 
for all $v \in W^{1}_{p}(\Omega)$, $\Omega\subset\mathbb R^2$. 

The problem of estimates of  $\mu_p^{(1)}(\Omega)$ is one of mainly interesting problems of the modern geometric analysis and its applications to the continuum  mechanics. 
The classical upper estimate for the first nontrivial Neumann eigenvalue of the Laplace operator   
$$
\mu_2^{(1)}(\Omega)\leq \mu_2^{(1)}(\Omega^{\ast})=\frac{p^2_{n/2}}{R^2_{\ast}}
$$
was proved by Szeg\"o \cite{S54} for simply connected planar domains and by Weinberger  \cite{W56} for  domains in $\mathbb{R}^n$. In this inequality $p_{n/2}$ denotes the first positive zero of the function $(t^{1-n/2}J_{n/2}(t))'$, and $\Omega^{\ast}$ is an $n$-ball of the same $n$-volume as $\Omega$ with $R_{\ast}$ as its radius. In particular, if $n=2$, we have $p_1=j_{1,1}'\approx1.84118$ where $j_{1,1}'$ denotes the first positive zero of the derivative of the Bessel function $J_1$. 

The upper estimates of the Laplace eigenvalues with the help of different techniques were intensively studied in the recent decades, see, for example, \cite{A98,AB93,AL97,EP15,LM98}.

The usual approach to lower estimates of the Laplace eigenvalues is based on the integral representations machinery in  convex domains. On this base lower estimates of first non-trivial eigenvalues for convex domains were given in terms of Euclidean diameter of the domains (see, for example, \cite{ENT,FNT,PW}). But Nikodim type examples \cite{M} show that in non-convex domains $\mu_p^{(1)}(\Omega)$ any estimates in terms of Euclidean diameters are not relevant.

We suggested in our previous works another type of estimates of in terms of integrals of conformal derivatives that can be reformulated in terms of hyperbolic radii of domains. So, we can say that hyperbolic geometry represents a natural language for the spectral properties of the Laplace operator. The integrals of conformal derivatives are not simple for analytical estimates, but if domains allow quasiconformal reflections \cite{Ahl,GR} we can simplify the problem and obtain the lower estimates of the principal frequency $\mu_p^{(1)}(\Omega)$ in terms of "quasiconformal" geometry of domains.

The main result of the paper is:

\vskip 0.2cm
\noindent
{\bf Theorem A.}
{\textit  Let $\Omega\subset\mathbb R^2$ be a $K$-quasidisc. Then
$$
\mu_p^{(1)}(\Omega)\geq \frac{M_p(K)}{|\Omega|^{\frac{p}{2}}}=\frac{M^{\ast}_p(K)}{R_{\ast}^{p}},
$$
where $R_{\ast}$ is a radius of a disc $\Omega^{\ast}$ of the same area as $\Omega$ and $M^{\ast}_p(K)=M_p(K)\pi^{-{p/2}}$.
}

\vskip 0.2cm

The quantity $M_p(K)$ in Theorem A depends on $p$ and a quasiconformity coefficient $K$ only:
\begin{multline*}
M_p(K)= \frac{\pi^{\frac{p}{2}}}{2^{p-2}K^2} \exp\left(-{\frac{K^2 \pi^2(2+ \pi^2)^2}{2\log3}}\right) \times \\
{} \inf_{2<\alpha< \alpha*} \inf_{1\leq q\leq 2} 
\left\{\left(\frac{1-\delta}{1/2-\delta}\right)^{(\delta-1)p} C_\alpha^{-2}\right\} , \\
{} C_\alpha=\frac{10^{6}}{[(\alpha -1)(1- \nu(\alpha))]^{1/\alpha}}, 
\end{multline*} 
where $\delta=1/q-(\alpha-2)/p\alpha$, $\alpha*=\min \left({\frac{K^2}{K^2-1}, \gamma*}\right)$, where $\gamma*$ is the unique solution of the equation 
$\nu(\alpha):=10^{4 \alpha}({(\alpha -2)}/{(\alpha -1)})(24\pi^2K^2)^{\alpha}=1$.

\begin{remark} The function 
$\nu(\alpha) $ is a monotone increasing function. Hence for any $\alpha<\alpha*$ the number $(1- \nu(\alpha))>0$ and $C_{\alpha}>0$.
\end{remark}

\begin{remark}
Recall that $K$-quasidiscs are images of the unit discs under $K$-quasicon\-for\-mal homeomorphisms of the plane $\mathbb R^2$. This class includes all Lipschitz simply connected domains but also includes a class of fractal domains (for example, so-called Rohde snowflakes \cite{Roh}). Hausdorff dimensions of quasidiscs boundary can be any number of $[1,2)$.
\end{remark}

Theorem A is based on the following theorem, which characterizes the Neumann eigenvalues in the terms of conformal derivatives. In previous works we introduced a concept of conformal $\alpha$-regular domains. Let $\varphi : \mathbb D \to \Omega$ be a conformal mapping from the unit disc $\mathbb D$ onto
bounded domain $\Omega$. The domain $\Omega$ is a conformal $\alpha$-regular for some $\infty\geq \alpha >2$ if $\| \varphi'\,|\,L_{\alpha}(\mathbb D)\|<\infty$.

\vskip 0.2cm
\noindent
{\bf Theorem B.}
\textit{Let $\varphi : \mathbb D \to \Omega$ be a conformal mapping from the unit disc $\mathbb D$ onto
conformal $\alpha$-regular domain $\Omega$.  
Then for any $p>2$ the following inequality holds:
$$
\frac{1}{\mu_p^{(1)}(\Omega)} 
 \leq C_p \cdot |\Omega|^{\frac{p-2}{2}} \cdot \| \varphi'\,|\,L_{\alpha}(\mathbb D)\|^2, 
$$
where 
$$
C_p=2^p \pi^{\frac{\alpha-2}{\alpha}-\frac{p}{2}} \inf\limits_{q \in [1, 2]} \left(\frac{1-\delta}{1/2-\delta}\right)^{(1-\delta)p}, \quad 
\delta=\frac{1}{q}-\frac{\alpha-2}{p\alpha}.
$$
}

\vskip 0.2cm

\begin{remark}
As an application of this result we obtain lower estimates of ${\mu_p^{(1)}(\Omega)}$ in the domains bounded by an epicycloid of $(n-1)$ cusps, which are non-convex domains.
\end{remark}

\noindent
{\bf Example.}
For $n \in \mathbb{N}$, the diffeomorphism 
$$
\varphi(z)=z+\frac{1}{n}z^n, \quad z=x+iy,
$$
is conformal and maps the unit disc $\mathbb D$ onto the domain $\Omega_n$ bounded by an epicycloid of $(n-1)$ cusps, inscribed in the circle $|w|=(n+1)/n$.
Since $\Omega_n$ is a conformal $\infty$-regular domain, then we have  
$$
\frac{1}{\mu_p^{(1)}(\Omega_n)} \leq 
2^{p+2} \left(\frac{n+1}{n}\right)^{p-2} \inf_{q \in [1, 2]} 
\left(\frac{1-\delta}{1/2-\delta}\right)^{(1-\delta)p} ,
$$
where $\delta=1/q-1/p$.

The Theorem B is based on the existence of the composition operator in Sobolev spaces
$$
\varphi^{\ast}: L^1_p(\Omega)\to L^1_q(\mathbb D),\,\,q<p,
$$
with the norm $\|\varphi^{\ast}\|\leq K_{p,q}(\mathbb D)$. In the case of conformal mappings $\varphi : \mathbb D \to \Omega$ and $1\leq q\leq 2< p<\infty$, we have
\begin{equation}
\label{confest}
\|\varphi^{\ast}\|\leq K_{p,q}(\mathbb D)= \left(\iint\limits_{\mathbb D} |\varphi'(x,y)|^\frac{(p-2)q}{p-q}~dxdy\right)^{\frac{p-q}{pq}}
\leq |\Omega|^{\frac{p-2}{2p}} \cdot \pi^{\frac{2-q}{2q}}.
\end{equation}

Therefore we can distinguish three different cases of estimates of the composition operators norm if they are generated by conformal mappings. The first case $p=q=2$ corresponds to the classical Laplace operator. Conformal mappings induces isometries of spaces $L^1_2(\Omega)$ and $L^1_2(\mathbb D)$ and as result we obtain estimates of a first non-trivial eigenvalue with the help of  Lebesgue norms of conformal derivatives in spaces $L_{\alpha}(\mathbb D)$ \cite{GU16} for $\alpha$-regular domains $\Omega$. The case $p<2$ corresponds to singular $p$-Laplace operators, then (\ref{confest}) is the singular integral and its convergence and estimates of first non-trivial eigenvalues depends on Brennan's Conjecture \cite{GPU17,GU2016} for composition operators.
The case $p>2$ corresponds to degenerate $p$-Laplace operators, the integral~(\ref{confest}) is finite for conformal $\alpha$-regular domains and  the inverse H\"older inequality permit us to estimate this integral for quasidiscs. The proposed approach permits us also to obtain spectral estimates of degenerate $p$-Laplace Neumann operator in quasidiscs (Theorem A) in terms of quasiconformal geometry. Theorem A will be illustrated by estimates of the first non-trivial eigenvalue of degenerate $p$-Laplace operator in  star-shaped and spiral-shaped domains. Reformulating the notion of quasidiscs in terms of Ahlfors's three-point condition we obtain Theorem C that gives estimates of the first non-trivial eigenvalues in terms of bounded turning condition. As a consequence we obtain the spectral estimates in snowflake like domains.

Recall one more time that a domain $\Omega$ is called a conformal $\alpha$-regular domain \cite{BGU} if $\varphi'\in L_{\alpha}(\mathbb D)$ for some $\alpha>2$. The degree $\alpha$ does not depends on choice of a conformal mapping $\varphi:\mathbb D\to\Omega$  (by the Riemann Mapping Theorem) and depends on the hyperbolic metric on $\Omega$ only.
A domain $\Omega$ is a conformal regular domain if it is an $\alpha$-regular domain for some $\alpha>2$. Note that any $C^2$-smooth simply connected bounded domain is $\infty$-regular (see, for example, \cite{Kr}).

A problem of exact constants in $(r,q)$-Poincar\'e-Sobolev inequalities is an open problem even in the unit disc.  We can use only	existing rough estimates of such constants in the case of convex domains \cite{GT77,GU16}.

Theorem B can be reformulated in terms of hyperbolic radii $R(\varphi(z),\Omega)$ \cite{Av12,BM07} of domains.

\vskip 0.2cm
\noindent
{\bf Theorem B*.}
\textit{Let $\varphi : \mathbb D \to \Omega$ be a conformal mapping from the unit disc $\mathbb D$ onto
conformal $\alpha$-regular domain $\Omega$.  
Then for any $p>2$ the following inequality holds:
\[
\frac{1}{\mu_p^{(1)}(\Omega)} 
\leq C_p \cdot |\Omega|^{\frac{p-2}{2}}
\left(\iint\limits_{\mathbb D}\left(\frac{R(\varphi(z),\Omega)}{1-|z|^2}\right)^{\alpha}~dxdy\right)^{\frac{2}{\alpha}}, 
\]
where 
$$
C_p=2^p \pi^{\frac{\alpha-2}{\alpha}-\frac{p}{2}} \inf\limits_{q \in [1, 2]} \left(\frac{1-\delta}{1/2-\delta}\right)^{(1-\delta)p}, \quad 
\delta=\frac{1}{q}-\frac{\alpha-2}{p\alpha}.
$$
}
\vskip 0.2cm
	
Few words about our machinery that is based  on the geometric theory of composition operators \cite{U93,VU02} and its applications to the Sobolev type embedding theorems \cite{GG94,GU09}. 

The following diagram roughly illustrates the main idea:

\[\begin{array}{rcl}
W^{1}_{p}(\Omega) & \stackrel{\varphi^*}{\longrightarrow} & W^{1}_{q}(\mathbb D) \\[2mm]
\multicolumn{1}{c}{\downarrow} & & \multicolumn{1}{c}{\downarrow} \\[1mm]
L_s(\Omega) & \stackrel{(\varphi^{-1})^*}{\longleftarrow} & L_r(\mathbb D).
\end{array}\]

Here the operator $\varphi^*$ defined by the composition rule
$\varphi^*(f)=f \circ \varphi$ is a bounded composition operator on Sobolev
spaces induced by a homeomorphism $\varphi$ of $\Omega$ and $\mathbb D$ and
the operator $(\varphi^{-1})^*$ defined by the composition rule
$(\varphi^{-1})^*(f)=f \circ \varphi^{-1}$ is a bounded composition operator on
Lebesgue spaces. This method allows to transfer Poincar\'e-Sobolev inequalities 
from regular domains (for example, from the unit disc $\mathbb D$) to $\Omega$.

\vskip 0.2cm

In the recent works we studied composition operators on Sobolev spaces defined on
planar domains in connection with the conformal mappings theory \cite{GU12}. This connection
leads to weighted Sobolev  embeddings \cite{GU13,GU14} with the universal conformal weights.
Another application of conformal composition operators was given in \cite{BGU} where the 
spectral stability problem for conformal regular domains was considered.

\section{Composition operators in $\alpha$-regular domains}

Let $\Omega$ be a domain in the Euclidean plane $\mathbb R^2$. For any $1\leq p<\infty$ we consider the Lebesgue space $L_p(\Omega)$ of measurable functions $f: \Omega \to \mathbb{R}$ equipped with the following norm:
\[
\|f\mid L_p(\Omega)\|=\biggr(\iint\limits _{\Omega}|f(x,y)|^{p}\, dxdy\biggr)^{\frac{1}{p}}<\infty.
\]  

We consider the Sobolev space $W^1_p(\Omega)$, $1\leq p<\infty$,
as a Banach space of locally integrable weakly differentiable functions
$f:\Omega\to\mathbb{R}$ equipped with the following norm: 
\[
\|f\mid W^1_p(\Omega)\|=\biggr(\iint\limits _{\Omega}|f(x,y)|^{p}\, dxdy\biggr)^{\frac{1}{p}}+
\biggr(\iint\limits _{\Omega}|\nabla f(x,y)|^{p}\, dxdy\biggr)^{\frac{1}{p}}.
\]
Recall that the Sobolev space $W^1_p(\Omega)$ coincides with the closure of the space of smooth functions $C^{\infty}(\Omega)$ in the norm of $W^1_p(\Omega)$.

We consider also the homogeneous seminormed Sobolev space $L^1_p(\Omega)$, $1\leq p<\infty$,
of locally integrable weakly differentiable functions $f:\Omega\to\mathbb{R}$ equipped
with the following seminorm: 
\[
\|f\mid L^1_p(\Omega)\|=\biggr(\iint\limits _{\Omega}|\nabla f(x,y)|^{p}\, dxdy\biggr)^{\frac{1}{p}}.
\]

\begin{remark}
By the standard definition functions of $L^1_p(\Omega)$ are defined only up to a set of measure zero, but they can be redefined quasieverywhere i.~e. up to a set of $p$-capacity zero. Indeed, every function $f\in L^1_p(\Omega)$ has a unique quasicontinuous representation $\tilde{f}\in L^1_p(\Omega)$. A function $\tilde{f}$ is termed quasicontinuous if for any $\varepsilon >0$ there is an open  set $U_{\varepsilon}$ such that the $p$-capacity of $U_{\varepsilon}$ is less than $\varepsilon$ and on the set $\Omega\setminus U_{\varepsilon}$ the function  $\tilde{f}$ is continuous (see, for example \cite{HKM,M}). 
\end{remark}

Let $\varphi:\Omega\to\widetilde{\Omega}$ be weakly differentiable in $\Omega$. The mapping $\varphi$ is the mapping of finite distortion if $|D\varphi(z)|=0$ for almost all $x\in Z=\{z\in\Omega : J(x,\varphi)=0\}$.

Let $\varphi:\Omega\to\widetilde{\Omega}$ be a homeomorphism. Then $\varphi$ is called a mapping of bounded $(p,q)$-distortion \cite{UV10}, if $\varphi\in W^1_{1,\loc}(\Omega)$, has finite distortion, and
$$
K_{p,q}(\Omega)=\left(\iint\limits_\Omega \left(\frac{|\varphi'(x,y)|^p}{|J_{\varphi}(x,y)|}\right)^\frac{q}{p-q}~dxdy\right)^\frac{p-q}{pq}<\infty.
$$

Classes of mappings of bounded $(p,q)$-distortion are closely connected with composition operators on Sobolev spaces.

Let $\Omega$ and $\widetilde{\Omega}$ be domains in $\mathbb R^2$. We say that
a diffeomorphism $\varphi:\Omega\to\widetilde{\Omega}$ induces a bounded composition
operator 
\[
\varphi^{\ast}:L^1_p(\widetilde{\Omega})\to L^1_q(\Omega),\,\,\,1\leq q\leq p\leq\infty,
\]
by the composition rule $\varphi^{\ast}(f)=f\circ\varphi$, if the composition $\varphi^{\ast}(f)\in L^1_q(\Omega)$
is defined quasi-everywhere in $\Omega$ and there exists a constant $K_{p,q}(\Omega)<\infty$ such that 
\[
\|\varphi^{\ast}(f)\mid L^1_q(\Omega)\|\leq K_{p,q}(\Omega)\|f\mid L^1_p(\widetilde{\Omega})\|
\]
for any function $f\in L^1_p(\widetilde{\Omega})$ \cite{VU04}.

The following theorem gives an analytic description of composition operators on Sobolev spaces:

\begin{theorem}
\label{CompTh} \cite{U93,VU02} A homeomorphism $\varphi:\Omega\to\widetilde{\Omega}$
between two domains $\Omega$ and $\widetilde{\Omega}$ induces a bounded composition
operator 
\[
\varphi^{\ast}:L^1_p(\widetilde{\Omega})\to L^1_q(\Omega),\,\,\,1\leq q< p<\infty,
\]
 if and only if $\varphi$ has finite distortion and is a mapping of bounded $(p,q)$-distortion.
The norm of the composition operator $\|\varphi^*\| \leq K_{p,q}(\Omega)$. 
\end{theorem}

Now we establish a connection between conformal $\alpha$-regular domains and the composition operators on Sobolev spaces.
\begin{theorem}
\label{compconf}
Let $\Omega \subset \mathbb R^2$ be a simply connected domain. Then $\Omega$ is a conformal $\alpha$-regular domain if and only if any
conformal mapping $\varphi: \mathbb D \to \Omega$ generates a bounded composition operator
\[
\varphi^*:L_p^1(\Omega) \to L_q^1(\mathbb D)
\]
for any $p \in (2\,, +\infty)$ and $q=p\alpha /(p+ \alpha -2)$.
\end{theorem}

\begin{proof}
By Theorem~\ref{CompTh}
$$
K_{p,q}(\mathbb D)=\left(\iint\limits_{\mathbb D} \left(\frac{|\varphi'(x,y)|^p}{J_{\varphi}(x,y)}\right)^\frac{q}{p-q}~dxdy\right)^\frac{p-q}{pq}<\infty
$$
if and only if a homeomorphism $\varphi: \mathbb D \to \Omega$ has finite distortion and induces a bounded composition operator
$$
\varphi^*: L_p^1(\Omega) \to L_q^1(\mathbb D), \quad 1\leq q<p<\infty.
$$

Let $\Omega$ is a conformal $\alpha$-regular domain. Since conformal mappings have finite distortion then for any conformal mapping $\varphi: \mathbb D \to \Omega$ 
$$
\iint\limits_{\mathbb D} |\varphi'(x,y)|^{\alpha}~dxdy< \infty \quad \text{for some} \quad \alpha >2.
$$

Using the conformal equality $|\varphi'(x,y)|^2= J_{\varphi}(x,y)>0$, we obtain
\begin{multline*}
K_{p,q}^{\frac{pq}{p-q}}(\mathbb D)= \iint\limits_{\mathbb D} \left(\frac{|\varphi'(x,y)|^p}{J_{\varphi}(x,y)}\right)^\frac{q}{p-q}~dxdy \\
{} = \iint\limits_{\mathbb D} |\varphi'(x,y)|^\frac{(p-2)q}{p-q}~dxdy = \iint\limits_{\mathbb D} |\varphi'(x,y)|^{\alpha}~dxdy< \infty
\end{multline*}
for $\alpha=(p-2)q/(p-q)$. Hence we have a bounded composition operator
$$
\varphi^*: L_p^1(\Omega) \to L_q^1(\mathbb D)
$$  
for any $p \in (2\,, +\infty)$ and $q=p\alpha/(p+\alpha -2)$.

Let us check that $q<p$. Because $p>2$ we have that $p+\alpha -2> \alpha >2$ and so $\alpha /(p+\alpha -2) <1$. Hence we obtain  $q<p$.

Suppose that the composition operator
$$
\varphi^*: L_p^1(\Omega) \to L_q^1(\mathbb D)
$$  
is bounded for any $p \in (2\,, +\infty)$ and $q=p\alpha/(p+\alpha -2)$. Then 
\begin{multline*}
\iint\limits_{\mathbb D} |\varphi'(x,y)|^{\alpha}~dxdy=
\iint\limits_{\mathbb D} |\varphi'(x,y)|^\frac{(p-2)q}{p-q}~dxdy \\
=
\iint\limits_{\mathbb D} \left(\frac{|\varphi'(x,y)|^p}{J_{\varphi}(x,y)}\right)^\frac{q}{p-q}~dxdy= K_{p,q}^{\frac{pq}{p-q}}(\mathbb D) < \infty.
\end{multline*}
\end{proof}

If $\Omega\subset\mathbb R^2$ is a conformal $\alpha$-regular domain, then by the Sobolev embedding theorem $\varphi$ belongs to the 
H\"older class $H^{\gamma}(\mathbb D)$, $\gamma=(\alpha-2)/\alpha$. Hence, the class of conformal regular domains allows description in terms of $\gamma$-hyperbolic boundary condition \cite{BP88}:
$$
\rho_{\Omega}\leq \frac{1}{\gamma}\log{\frac{\dist(z_0,\partial\Omega)}{\dist(z,\partial\Omega)}}+C_0, \quad z=(x,y),
$$
where $\rho_{\Omega}$ is the hyperbolic metric in $\Omega$.

Note, that if $\Omega$ is a conformal $\alpha$-regular domain, then it is a domain with $\gamma$-hyperbolic boundary condition for $\gamma=(\alpha-2)/\alpha$. Inverse, if $\Omega$ is a domain with $\gamma$-hyperbolic boundary condition, then $\Omega$ is a conformal regular domain for some $\alpha$, but calculation of $\gamma$ in terms of  $\alpha$ is a non solved problem. For our study we need the exact value of $\alpha$.

Theorem~\ref{compconf} implies:

\begin{corollary}
\label{cor:hyp}
Let $\Omega \subset \mathbb R^2$ be a simply connected domain. Then $\Omega$ satisfies a $\gamma$-hyperbolic boundary condition if and only if any conformal mapping $\varphi: \mathbb D \to \Omega$ generates a bounded composition operator
\[
\varphi^*:L_p^1(\Omega) \to L_q^1(\mathbb D)
\]
for any $p \in (2\,, +\infty)$ and some $q=q(p,\gamma)>2$.
\end{corollary}

We define the geodesic diameter $\diam_G(\Omega)$ of a domain $\Omega\subset\mathbb{R}^n$ as
$$
\diam_G(\Omega)=\sup\limits_{x,y\in\Omega}\dist_{\Omega}(x,y).
$$
Here $\dist_{\Omega}(x,y)$ is the intrinsic geodesic distance:
$$
\dist_{\Omega}(x,y)=\inf_{\gamma\in\Omega}\int\limits_0^1 |\gamma'(t)|~dt
$$
where infimum is taken over all rectifiable curves $\gamma\in\Omega$ such that $\gamma(0)=x$ and $\gamma(1)=y$.

Using \cite{GU14A} and Corollary~\ref{cor:hyp} we obtain a simple necessary geometric condition for domains with $\gamma$-hyperbolic boundary condition.

\begin{theorem}
Let $\Omega \subset \mathbb R^2$ be a simply connected domain. If  $\Omega$ satisfies a $\gamma$-hyperbolic boundary condition, then $\Omega$ has a finite geodesic diameter.
\end{theorem}

Note, that this theorem gives a simple proof that "maze-like" domain \cite{KOT} does not satisfies the $\gamma$-hyperbolic boundary condition, because this domain obviously has infinite geodesic diameter.

\begin{figure}[h!]
\centering
\includegraphics[width=0.3\textwidth]{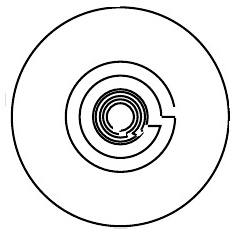}
\caption{A "maze-like" domain.}
\end{figure}

\section{Poincar\'e-Sobolev inequalities}

\textbf{Two-weight Poincar\'e-Sobolev inequalities}. Let $\Omega \subset \mathbb R^2$
be a planar domain and let $h : \Omega \to \mathbb R$ be a real valued locally integrable function such that $h(x)>0$ a.~e. in $\Omega$. We consider the weighted Lebesgue space $L_p(\Omega,h)$, $1\leq p<\infty$ is the space 
of measurable functions $f: \Omega \to \mathbb R$  with the finite norm
$$
\|f\,|\,L_{p}(\Omega,h)\|:= \left(\iint\limits_\Omega|f(x)|^ph(x,y)~dxdy \right)^{\frac{1}{p}}< \infty.
$$
It is a Banach space for the norm $\|f\,|\,L_{p}(\Omega,h)\|$.

In the following theorem we obtain the estimate of the norm of the composition operator on Sobolev spaces
in any simply connected domain with finite measure.

\begin{theorem}
\label{compconform}
Let $\Omega \subset \mathbb R^2$ be a simply connected domain with finite measure. Then 
conformal mapping $\varphi: \mathbb D \to \Omega$ generates a bounded composition operator
\[
\varphi^*:L_p^1(\Omega) \to L_q^1(\mathbb D)
\]
for any $p \in (2\,, +\infty)$ and $q \in [1, 2]$.
\end{theorem} 

\begin{proof}
By Theorem~\ref{CompTh} a homeomorphism $\varphi: \mathbb D \to \Omega$ induces a bounded composition operator
\[
\varphi^{\ast}:L^1_p(\mathbb D)\to L^1_q(\Omega),\,\,\,1\leq q<p<\infty,
\]
if and only if $\varphi \in W_{1,\loc}^1(\Omega)$, has finite distortion and
$$
K_{p,q}(\mathbb D)=\left(\iint\limits_{\mathbb D} \left(\frac{|\varphi'(x,y)|^p}{J_{\varphi}(x,y)}\right)^\frac{q}{p-q}~dxdy\right)^\frac{p-q}{pq}<\infty.
$$
Because $\varphi$ is a conformal mapping, then $\varphi$ have finite distortion. 

Using the conformal equality $|\varphi'(x,y)|^2= J_{\varphi}(x,y)>0$ we obtain
\begin{multline*}
K_{p,q}(\mathbb D)
= \left(\iint\limits_{\mathbb D} \left(\frac{|\varphi'(x,y)|^p}{J_{\varphi}(x,y)}\right)^\frac{q}{p-q}~dxdy\right)^{\frac{p-q}{pq}} \\
{} = \left(\iint\limits_{\mathbb D} |\varphi'(x,y)|^\frac{(p-2)q}{p-q}~dxdy\right)^{\frac{p-q}{pq}}.
\end{multline*}
Note that if $q\leq2$ then the quantity $(p-2)q/(p-q)\leq2$. Hence applying H\"older inequality to the last integral we get
\begin{multline*}
\left(\iint\limits_{\mathbb D} |\varphi'(x,y)|^\frac{(p-2)q}{p-q}~dxdy\right)^{\frac{p-q}{pq}} \\
{} \leq \left[\left(\iint\limits_{\mathbb D} |\varphi'(x,y)|^2~dxdy\right)^\frac{(p-2)q}{2(p-q)} \left(\iint\limits_{\mathbb D} dxdy\right)^\frac{(2-q)p}{2(p-q)}\right]^{\frac{p-q}{pq}}.
\end{multline*}
By the condition of the theorem, the domain $\Omega$ is a simply connected with finite measure therefore
$$
K_{p,q}(\mathbb D) \leq |\Omega|^{\frac{p-2}{2p}} \cdot \pi^{\frac{2-q}{2q}}< \infty.
$$
We proved that a composition operator
$$
\varphi^*: L_p^1(\Omega) \to L_q^1(\mathbb D)
$$  
is bounded for any $p \in (2\,, +\infty)$ and $q \in [1, 2]$.

\end{proof}

On the base of this theorem we prove existence of universal two-weight Poincar\'e-Sobolev inequalities in any simply connected domain $\Omega\subset\mathbb R^2$ with finite measure.

\begin{theorem}\label{UnivEst}
Let $\Omega\subset\mathbb R^2$ be a simply connected domain with finite measure and $h(u,v) =J_{\varphi^{-1}}(u,v)$ is 
the conformal weight defined by a conformal mapping 
$\varphi : \mathbb D \to \Omega$. Then for every function $f \in W^{1}_{p}(\Omega)$, $p>2,$
the inequality
\[
\inf\limits_{c \in \mathbb R}\left(\iint\limits_\Omega |f(u,v)-c|^rh(u,v)~dudv\right)^{\frac{1}{r}} \leq B_{r,p}(\Omega,h)
\left(\iint\limits_\Omega |\nabla f(u,v)|^p~dudv\right)^{\frac{1}{p}}
\]
holds for any $r \geq 1$ with the constant
$$
B_{r,p}(\Omega,h) \leq \inf\limits_{q \in [1, 2]} \left\{B_{r,q}(\mathbb D) \cdot \pi^{\frac{2-q}{2q}}\right\}\cdot |\Omega|^{\frac{p-2}{2p}}.
$$
\end{theorem}

Here $B_{r,q}(\mathbb D)$ is the best constant in the (non-weighted) Poincar\'e-Sobolev inequality for the unit disc
$\mathbb D \subset \mathbb R^2$.

\begin{proof}
Let $r \geq 1$. By the Riemann Mapping Theorem there exists a conformal mapping $\varphi : \mathbb D \to \Omega$.
Denote by $h(u,v): =J_{\varphi^{-1}}(u,v)$ the conformal weight in $\Omega$.

Using the change of variable formula for conformal mapping, the classical Poincar\'e-Sobolev inequality for the unit disc $\mathbb D \subset \mathbb R^2$ 
$$
\inf\limits_{c \in \mathbb R}\left(\iint\limits_{\mathbb D} |g(x,y)-c|^r~dxdy\right)^{\frac{1}{r}} \leq B_{r,q}(\mathbb D)
\left(\iint\limits_{\mathbb D} |\nabla g(x,y)|^q~dxdy\right)^{\frac{1}{q}}
$$
and Theorem~\ref{compconf}, we get for a smooth function $g\in W^1_p (\Omega)$
\begin{multline*}
\inf\limits_{c \in \mathbb R}\left(\iint\limits_\Omega |f(u,v)-c|^rh(u,v)~dudv\right)^{\frac{1}{r}} =
\inf\limits_{c \in \mathbb R}\left(\iint\limits_\Omega |f(u,v)-c|^rJ_{\varphi^{-1}}(u,v)~dudv\right)^{\frac{1}{r}} \\
{} = \inf\limits_{c \in \mathbb R}\left(\iint\limits_{\mathbb D} |g(x,y)-c|^r~dxdy\right)^{\frac{1}{r}} \leq B_{r,q}(\mathbb D)
\left(\iint\limits_{\mathbb D} |\nabla g(x,y)|^q~dxdy\right)^{\frac{1}{q}} \\
{} \leq B_{r,q}(\mathbb D) \cdot \pi^{\frac{2-q}{2q}} \cdot |\Omega|^{\frac{p-2}{2p}} \left(\iint\limits_{\Omega} |\nabla f(u,v)|^p~dudv\right)^{\frac{1}{p}}.
\end{multline*}

Approximating an arbitrary function $f \in W^{1}_{p}(\Omega)$ by smooth functions we have
\[
\inf\limits_{c \in \mathbb R}\left(\iint\limits_\Omega |f(u,v)-c|^rh(u,v)~dudv\right)^{\frac{1}{r}} \leq B_{r,p}(\Omega,h)
\left(\iint\limits_\Omega |\nabla f(u,v)|^p~dudv\right)^{\frac{1}{p}}
\]
with the constant
$$
B_{r,p}(\Omega,h) \leq \inf\limits_{q \in [1, 2]} \left\{B_{r,q}(\mathbb D) \cdot \pi^{\frac{2-q}{2q}}\right\}\cdot |\Omega|^{\frac{p-2}{2p}}.
$$
\end{proof}

The property of the conformal $\alpha$-regularity implies the integrability of a 
Jacobian of conformal mappings and therefore for any conformal $\alpha$-regular domain we have the embedding of weighted Lebesgue spaces $L_r(\Omega,h)$ into non-weighted Lebesgue
spaces $L_s(\Omega)$ for $s=\frac{\alpha -2}{\alpha}r$:
\begin{lemma} \label{Le3.3}
Let $\Omega$ be a conformal $\alpha$-regular domain.Then for any function
$f \in L_r(\Omega,h)$, $\alpha / (\alpha - 2) \leq r < \infty$, the inequality
$$
\|f\,|\,L_s(\Omega)\| \leq \left(\iint\limits_\mathbb D \big|\varphi'(x,y)\big|^{\alpha}~dxdy \right)^{{\frac{2}{\alpha}} \cdot \frac{1}{s}} ||f\,|\,L_r(\Omega,h)||
$$
holds for $s=\frac{\alpha -2}{\alpha}r$.
\end{lemma}

\begin{proof}
Because $\Omega$ is a conformal $\alpha$-regular domain then for any conformal mapping $\varphi: \mathbb D \to \Omega$ we have
$$
\left(\iint\limits_\mathbb D J_{\varphi}^{\frac{r}{r-s}}(x,y)~dxdy\right)^{\frac{r-s}{rs}} =
\left(\iint\limits_\mathbb D \big|\varphi'(x,y)\big|^{\alpha}~dxdy\right)^{{\frac{2}{\alpha}} \cdot \frac{1}{s}}  < +\infty,
$$
for $s=\frac{\alpha -2}{\alpha}r$. 
Using the change of variable formula for conformal mappings, H\"older's inequality with exponents $(r,rs/(r-s))$ and the conformal weight
$h(u,v):=J_{\varphi^{-1}}(u,v) $, we get
\begin{multline*}
\|f\,|\,L_s(\Omega)\| \\
{} = \left(\iint\limits_{\Omega} |f(u,v)|^s dudv\right)^{\frac{1}{s}} = 
\left(\iint\limits_{\Omega} |f(u,v)|^s J_{\varphi^{-1}}^{\frac{s}{r}}(u,v) J_{\varphi^{-1}}^{-\frac{s}{r}}(u,v)~dudv\right)^{\frac{1}{s}} \\
{} \leq \left(\iint\limits_{\Omega} |f(u,v)|^r J_{\varphi^{-1}}(u,v) dudv\right)^{\frac{1}{r}}
 \left(\iint\limits_{\Omega} J_{\varphi^{-1}}^{- \frac{s}{r-s}}(u,v)~ dudv\right)^{\frac{r-s}{rs}} \\
{} \leq \left(\iint\limits_{\Omega} |f(u,v)|^r h(u,v)~dudv\right)^{\frac{1}{r}} 
\left(\iint\limits_{\mathbb D} J_{\varphi}^{\frac{r}{r-s}}(x,y)~dxdy\right)^{\frac{r-s}{rs}} \\
{} = \left(\iint\limits_{\Omega} |f(u,v)|^r h(u,v)~dudv\right)^{\frac{1}{r}}
\left(\iint\limits_\mathbb D \big|\varphi'(x,y)\big|^{\alpha}~dxdy\right)^{{\frac{2}{\alpha}} \cdot \frac{1}{s}}.
\end{multline*}
 
\end{proof}

The following theorem gives the upper estimate of the Poincar\'e-Sobolev constant as an application of Theorem \ref{UnivEst} and Lemma \ref{Le3.3}:

\begin{theorem}\label{UEPC}
Let $\Omega \subset \mathbb R^2$ be a conformal $\alpha$-regular domain. Then for any function $f \in W^1_p(\Omega)$, $p>2$, the  
Poincar\'e-Sobolev inequality
\[
\inf\limits_{c \in \mathbb R}\left(\iint\limits_\Omega |f(u,v)-c|^s~dudv\right)^{\frac{1}{s}} \leq B_{s,p}(\Omega)
\left(\iint\limits_\Omega |\nabla f(u,v)|^p~dudv\right)^{\frac{1}{p}}
\]
holds for any $s \geq 1$ with the constant
\begin{multline*}
B_{s,p}(\Omega) \leq \left(\iint\limits_\mathbb D \big|\varphi'(x,y)\big|^{\alpha}~dxdy\right)^{{\frac{2}{\alpha}} \cdot \frac{1}{s}} B_{r,p}(\Omega, h) \\
{} \leq \inf\limits_{q \in [1, 2]} \left\{B_{\frac{\alpha s}{\alpha-2},q}(\mathbb D) \cdot \pi^{\frac{2-q}{2q}}\right\}\cdot |\Omega|^{\frac{p-2}{2p}} 
\cdot \| \varphi'\,|\,L_{\alpha}(\mathbb D)\|^{\frac{2}{s}}.
\end{multline*}
\end{theorem}

\begin{proof}
Let $f \in W^1_p(\Omega)$, $p>2$. Then by Theorem~\ref{UnivEst} and Lemma~\ref{Le3.3} we get 
\begin{multline*}
\inf\limits_{c \in \mathbb R}\left(\iint\limits_\Omega |f(u,v)-c|^s~dudv\right)^{\frac{1}{s}} \\
{} \leq  \left(\iint\limits_\mathbb D \big|\varphi'(x,y)\big|^{\alpha}~dxdy\right)^{{\frac{2}{\alpha}} \cdot \frac{1}{s}}
\inf\limits_{c \in \mathbb R}\left(\iint\limits_\Omega |f(u,v)-c|^rh(u,v)~dudv\right)^{\frac{1}{r}} \\
{} \leq B_{r,p}(\Omega, h) \left(\iint\limits_\mathbb D \big|\varphi'(x,y)\big|^{\alpha}~dxdy\right)^{{\frac{2}{\alpha}} \cdot \frac{1}{s}}
\left(\iint\limits_\Omega |\nabla f(u,v)|^p~dudv\right)^{\frac{1}{p}} 
\end{multline*}
for $s \geq 1$.

Because by Lemma~\ref{Le3.3} $s=\frac{\alpha -2}{\alpha}r$ and by Theorem~\ref{UnivEst} $r \geq 1$, then $s \geq 1$.
\end{proof}

By the generalized version of Rellich-Kondrachov compactness theorem (see, for example, \cite{M} or \cite{HK}) and the $(r,p)$--Sobolev-Poincar\'e inequality for $r>p$ follows that 
the embedding operator
$$
i: W^1_p(\Omega) \hookrightarrow L_p(\Omega)
$$
is compact in conformal $\alpha$-regular domains.

Hence, the first non-trivial Neumann eigenvalue $\mu_p^{(1)}(\Omega)$ can be characterized \cite{ENT} as
$$
\mu_p^{(1)}(\Omega)=\min \left\{\frac{\iint\limits_\Omega |\nabla u(x,y)|^{p}~dxdy}{\iint\limits_\Omega |u(x,y)|^{p}~dxdy} : u \in W_p^1(\Omega) \setminus \{0\}, 
\iint\limits_\Omega |u|^{p-2}u~dxdy=0 \right\}.
$$

Moreover, $\mu_p^{(1)}(\Omega)^{-\frac{1}{p}}$ is the best constant $B_{p,p}(\Omega)$ ( see, for example, \cite{GU2016}) in the following Poincar\'e-Sobolev inequality
\[
\inf\limits_{c \in \mathbb R}\left(\iint\limits_\Omega |f(x,y)-c|^p~dxdy\right)^{\frac{1}{p}} \leq B_{p,p}(\Omega)
\left(\iint\limits_\Omega |\nabla f(x,y)|^p~dxdy\right)^{\frac{1}{p}}, \,\, f \in W_p^1(\Omega).
\]

Theorem~\ref{UEPC} in case $s=p$ implies the lower estimates of the first non-trivial eigenvalue of the degenerate $p$-Laplace Neumann
operator in conformal $\alpha$-regular domains $\Omega \subset \mathbb R^2$ with finite measure.

\vskip 0.2cm
\noindent
{\bf Theorem B.}
\textit{Let $\varphi : \mathbb D \to \Omega$ be a conformal mapping from the unit disc $\mathbb D$ onto
conformal $\alpha$-regular domain $\Omega$.  
Then for any $p>2$ the following inequality holds
$$
\frac{1}{\mu_p^{(1)}(\Omega)} 
 \leq C_p \cdot |\Omega|^{\frac{p-2}{2}} \cdot \| \varphi'\,|\,L_{\alpha}(\mathbb D)\|^2, 
$$
where 
$$
C_p=2^p \pi^{\frac{\alpha-2}{\alpha}-\frac{p}{2}} \inf\limits_{q \in [1, 2]} \left(\frac{1-\delta}{1/2-\delta}\right)^{(1-\delta)p}, \quad 
\delta=\frac{1}{q}-\frac{\alpha-2}{p\alpha}.
$$
}

\vskip 0.2cm
In case conformal $\infty$-regular domains we have the following assertion:
\begin{corollary}\label{Cor3.4}
Let $\varphi : \mathbb D \to \Omega$ be a conformal mapping from the unit disc $\mathbb D$ onto
conformal $\infty$-regular domain $\Omega$. 
Then for any $p>2$ the following inequality holds
$$
\frac{1}{\mu_p^{(1)}(\Omega)} 
 \leq C_p \cdot |\Omega|^{\frac{p-2}{2}} \cdot \| \varphi'\,|\,L_{\infty}(\mathbb D)\|^2, 
$$
where 
$$
C_p=2^p \pi^{1-\frac{p}{2}} \inf\limits_{q \in [1, 2]} \left(\frac{1-\delta}{1/2-\delta}\right)^{(1-\delta)p}, \quad 
\delta=\frac{1}{q}-\frac{1}{p}.
$$
\end{corollary} 

\vskip 0.2cm
As examples, we consider the domains bounded by an epicycloid. Since the domains bounded by an epicycloid are conformal $\infty$-regular domains, we can apply Corollary~\ref{Cor3.4}, i.e.:

\begin{example}
For $n \in \mathbb{N}$, the diffeomorphism 
$$
\varphi(z)=z+\frac{1}{n}z^n, \quad z=x+iy,
$$
is conformal and maps the unit disc $\mathbb D$ onto the domain $\Omega_n$ bounded by an epicycloid of $(n-1)$ cusps, inscribed in the circle $|w|=(n+1)/n$.
The image of $\varphi$ for $n=2$, $n=5$ and $n=8$ is illustrated in Figure 2.

Now we estimate the norm of the complex derivative $\varphi'$ in $L_{\infty}(\mathbb D)$ and the area of domain $\Omega_n$. A straightforward calculation yields
$$
\|\varphi'\,|\,L_{\infty}(\mathbb D)\|=\esssup\limits_{|z|\leq 1}(|1+z^{n-1}|)\leq 2
$$
and 
\[
|\Omega_n| \leq \pi \left(\frac{n+1}{n}\right)^2.
\]
Then by Corollary~\ref{Cor3.4} we have 
$$
\frac{1}{\mu_p^{(1)}(\Omega_n)} \leq 
2^{p+2} \left(\frac{n+1}{n}\right)^{p-2} \inf_{q \in [1, 2]} 
\left(\frac{1-\delta}{1/2-\delta}\right)^{(1-\delta)p} ,
$$
where $\delta=1/q-1/p$.
\end{example}

\begin{figure}[h!]
\centering
\includegraphics[width=0.6\textwidth]{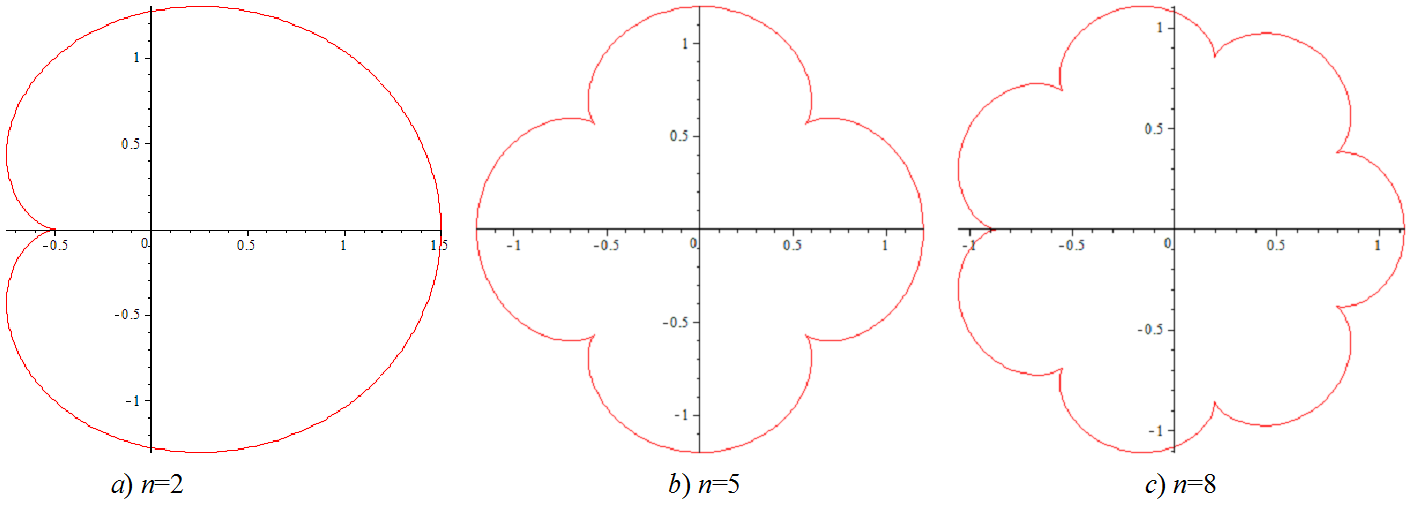}
\caption{Image of $\mathbb{D}$ under $\varphi(z)=z+\frac{1}{n}z^n$.}
\end{figure}

\section{Spectral estimates in quasidiscs}

Recall that a domain $\Omega\subset\mathbb R^2$ is called a $K$-quasidisc if it is the image of the unit disc $\mathbb D$ under a $K$-quasiconformal mapping of the plane $\mathbb R^2$ onto itself. Note that quasidiscs represent large class domains including fractal type domains like snowflakes.

In this section we obtain estimates of integrals of conformal derivatives in  quasidiscs.

Follow \cite{Ahl66} a homeomorphism $\varphi:\Omega \rightarrow \Omega'$
between planar domains is called $K$-quasiconformal if it preserves
orientation, belongs to the Sobolev class $W_{2,\loc}^{1}(\Omega)$
and its directional derivatives $\partial_{\xi}$ satisfy the distortion inequality
$$
\max\limits_{\xi}|\partial_{\xi}\varphi|\leq K\min_{\xi}|\partial_{\xi}\varphi|\,\,\,
\text{a.e. in}\,\,\, \Omega \,.
$$

For any planar $K$-quasiconformal homeomorphism $\varphi:\Omega\rightarrow \Omega'$
the following sharp result is known: $J(z,\varphi)\in L_{p,\loc}(\Omega)$
for any $1 \leq p<\frac{K}{K-1}$ (\cite{A94,G81}). Hence for any conformal mapping $\varphi:\mathbb{D}\to\Omega$
of the unit disc $\mathbb{D}$ onto a $K$-quasidisc $\Omega$ its derivatives $\varphi'\in L_p(\mathbb{D})$ for any 
$1\le p<\frac{2K^{2}}{K^2-1}$~\cite{GU16}.

Using integrability of conformal derivatives on the base of the weak inverse H\"older inequality and the measure doubling condition \cite{GPU17_2} we obtain an estimate of the constant in the inverse H\"older inequality for Jacobians of quasiconformal mappings. The following theorem was proved but not formulated in \cite{GPU17_2}.

\vskip 0.2cm

\begin{theorem}
\label{thm:IHIN}
Let $\varphi:\mathbb R^2 \to \mathbb R^2$ be a $K$-quasiconformal mapping. Then for every disc $\mathbb D \subset \mathbb R^2$ and
for any $1<\kappa<\frac{K}{K-1}$ the inverse H\"older inequality
\begin{equation*}\label{RHJ}
\left(\iint\limits_{\mathbb D} |J_{\varphi}(x,y)|^{\kappa}~dxdy \right)^{\frac{1}{\kappa}}
\leq \frac{C_\kappa^2 K \pi^{\frac{1}{\kappa}-1}}{4}
\exp\left\{{\frac{K \pi^2(2+ \pi^2)^2}{2\log3}}\right\}\iint\limits_{\mathbb D} |J_{\varphi}(x,y)|~dxdy.
\end{equation*}
holds. Here
$$
C_\kappa=\frac{10^{6}}{[(2\kappa -1)(1- \nu)]^{1/2\kappa}}, \quad \nu = 10^{8 \kappa}\frac{2\kappa -2}{2\kappa -1}(24\pi^2K)^{2\kappa}<1.
$$
\end{theorem}

If $\Omega$ is a $K$-quasidisc, then a conformal mapping $\varphi: \mathbb D\to\Omega$ allows $K^2$-quasiconformal reflection \cite{Ahl}. Hence, by Theorem~\ref{thm:IHIN} we obtain the following integral estimates of complex derivatives of conformal mapping $\varphi:\mathbb D\to\Omega$ of the unit disc onto a $K$-quasidisc $\Omega$:

\begin{corollary}\label{Est_Der}
Let $\Omega\subset\mathbb R^2$ be a $K$-quasidisc and $\varphi:\mathbb D\to\Omega$ be a conformal mapping. Suppose that  $2<\lambda<\frac{2K^2}{K^2-1}$.
Then 
\begin{equation}\label{Ineq_2}
\left(\iint\limits_{\mathbb D} |\varphi'(x,y)|^{\lambda}~dxdy \right)^{\frac{1}{\lambda}}
\leq \frac{C_\lambda K \pi^{\frac{2-\lambda}{2 \lambda}}}{2}
\exp\left\{{\frac{K^2 \pi^2(2+ \pi^2)^2}{4\log3}}\right\}\cdot |\Omega|^{\frac{1}{2}}.
\end{equation}
where
$$
C_\lambda=\frac{10^{6}}{[(\lambda -1)(1- \nu)]^{1/\lambda}}, \quad \nu = 10^{4 \lambda}\frac{\lambda -2}{\lambda -1}(24\pi^2K^2)^{\lambda}<1.
$$
\end{corollary}

Combining Theorem~B and Corollary~\ref{Est_Der} we obtain spectral estimates of the degenerate $p$-Laplace operator with the Neumann boundary condition:

\vskip 0.2cm
\noindent
{\bf Theorem A.}
{\textit  Let $\Omega\subset\mathbb R^2$ be a $K$-quasidisc. Then
$$
\mu_p^{(1)}(\Omega)\geq \frac{M_p(K)}{|\Omega|^{\frac{p}{2}}}=\frac{M^{\ast}_p(K)}{R_{\ast}^{p}},
$$
where $R_{\ast}$ is a radius of a disc $\Omega^{\ast}$ of the same area as $\Omega$ and $M^{\ast}_p(K)=M_p(K)\pi^{-{p/2}}$.
}

\vskip 0.2cm

\begin{proof}
Quasidiscs are conformal $\alpha$-regular domains for $2<\alpha<\frac{2K^2}{K^2-1}$ \cite{GU16}. Then by Theorem~B for any $2<\alpha<\frac{2K^2}{K^2-1}$ we have
\begin{equation}\label{Est_1}
\frac{1}{\mu_p^{(1)}(\Omega)} 
 \leq C_p \cdot |\Omega|^{\frac{p-2}{2}} \cdot \| \varphi'\,|\,L_{\alpha}(\mathbb D)\|^2, 
\end{equation}
where 
$$
C_p=2^p \pi^{\frac{\alpha-2}{\alpha}-\frac{p}{2}} \inf\limits_{q \in [1, 2]} \left(\frac{1-\delta}{1/2-\delta}\right)^{(1-\delta)p}, \quad 
\delta=\frac{1}{q}-\frac{\alpha-2}{p\alpha}.
$$ 

Now we estimate integral from the right-hand side of this inequality. According to Corollary~\ref{Est_Der} we obtain
\begin{multline}\label{Est_2}
\|\varphi'\,|\,L_\alpha(\mathbb D)\|^2=
\left(\iint\limits_{\mathbb D} |\varphi'(x,y)|^\alpha~dxdy\right)^{\frac{1}{\alpha}\cdot 2} \\
{} \leq \frac{C_\alpha^2 K^2 \pi^{\frac{2}{\alpha}-1}}{4}
\exp\left\{{\frac{K^2 \pi^2(2+ \pi^2)^2}{2\log3}}\right\}\cdot |\Omega|.
\end{multline}

Combining inequalities~\eqref{Est_1} and \eqref{Est_2} 
we get the required inequality.
\end{proof}

As an application of Theorem~A we obtain the lower estimates of the first non-trivial eigenvalue on the 
Neumann eigenvalue problem for the degenerate $p$-Laplace operator in the star-shaped and spiral-shaped domains.

\vskip 0.2cm

\textbf{Star-shaped domains.} We say that a domain $\Omega^*$ is $\beta$-star-shaped (with respect to $z_0=0$)
if the function $\varphi(z)$, $\varphi(0)=0$, conformally maps a unit disc $\mathbb{D}$ onto $\Omega^*$ and the condition satisfies \cite{FKZ,Sug}:
\[
\left|\arg \frac{z \varphi^{\prime}(z)}{\varphi(z)} \right| \leq \beta \pi/2, \quad 0 \leq \beta <1, \quad |z|<1.
\]

In \cite{FKZ} proved the following: 
the boundary of the $\beta$-star-shaped domain $\Omega^*$ is a $K$-quasicircle with $K=\cot ^2(1-\beta)\pi/4$.

Then by Theorem~A we have
\begin{multline*}
\frac{1}{\mu_p^{(1)}(\Omega^*)} \leq 
{} \inf_{\alpha \in \left(2,\frac{2}{1-\tan ^4(1-\beta) \frac{\pi}{4}}\right)} \inf_{q \in [1, 2]} 
\left(\frac{1-\delta}{1/2-\delta}\right)^{(1-\delta)p}  \\
{} \times \frac{2^{p-2} C_\alpha^2 \cot ^4(1-\beta)\frac{\pi}{4}}{\pi^{\frac{p}{2}}}
\exp\left\{{\frac{\pi^2(2+ \pi^2)^2 \cot ^4(1-\beta)\frac{\pi}{4}}{2\log3}}\right\} \cdot \big|\Omega^*\big|^{\frac{p}{2}},
\end{multline*} 
where $\delta=1/q-(\alpha-2)/p\alpha$,
$$
C_\alpha=\frac{10^{6}}{[(\alpha -1)(1- \nu)]^{1/\alpha}}, \quad \nu = 10^{4 \alpha}\frac{\alpha -2}{\alpha -1}(24\pi^2 \cot ^4(1-\beta)\pi/4)^{\alpha}<1.
$$

\vskip 0.2cm

\textbf{Spiral-shaped domains.} We say that a domain $\Omega_s$ is $\beta$-spiral-shaped (with respect to $z_0=0$)
if the function $\varphi(z)$, $\varphi(0)=0$, conformally maps a unit disc $\mathbb{D}$ onto $\Omega_s$ and the condition satisfies \cite{S87,Sug}:
\[
\left|\arg e^{i \gamma} \frac{z \varphi^{\prime}(z)}{\varphi(z)} \right| \leq \beta \pi/2, \quad 0 \leq \beta <1, \quad |\gamma|<\beta \pi/2, \quad |z|<1.
\]

In \cite{S87} proved the following: the boundary of the $\beta$-spiral-shaped domain $\Omega_s$ is a $K$-quasicircle with $K=\cot ^2(1-\beta)\pi/4$.

Then by Theorem~A we have
\begin{multline*}
\frac{1}{\mu_p^{(1)}(\Omega_s)} \leq 
{} \inf_{\alpha \in \left(2,\frac{2}{1-\tan ^4(1-\beta) \frac{\pi}{4}}\right)} \inf_{q \in [1, 2]} 
\left(\frac{1-\delta}{1/2-\delta}\right)^{(1-\delta)p}  \\
{} \times \frac{2^{p-2} C_\alpha^2 \cot ^4(1-\beta)\frac{\pi}{4}}{\pi^{\frac{p}{2}}}
\exp\left\{{\frac{\pi^2(2+ \pi^2)^2 \cot ^4(1-\beta)\frac{\pi}{4}}{2\log3}}\right\} \cdot  \big|\Omega_s\big|^{\frac{p}{2}},
\end{multline*} 
where $\delta=1/q-(\alpha-2)/p\alpha$,
$$
C_\alpha=\frac{10^{6}}{[(\alpha -1)(1- \nu)]^{1/\alpha}}, \quad \nu = 10^{4 \alpha}\frac{\alpha -2}{\alpha -1}(24\pi^2 \cot ^4(1-\beta)\pi/4)^{\alpha}<1.
$$

\vskip 0.2cm

In order to obtain the lower estimates of the first non-trivial eigenvalue on the 
Neumann eigenvalue problem for the degenerate $p$-Laplace operator in fractal type domains we use description of quasidiscs in the terms of the Ahlfors's 3-point condition: {\it a Jordan curve $\Gamma$ satisfies the Ahlfors's 3-point condition: there exists a constant $C$ such that
\begin{equation}\label{A-cond}
|\zeta_3-\zeta_1| \leq C|\zeta_2-\zeta_1|, \quad C \geq 1
\end{equation}
for any three points on $\Gamma$, where $\zeta_3$ is between $\zeta_1$ and $\zeta_2$.}

In \cite{GPU17_2} was proved (Theorem~5.1) that if a domain $\Omega$ is bounded by Jordan curve $\Gamma$ satisfies the Ahlfors's 3-point condition, then a conformal mapping $\varphi:\mathbb D\to\Omega$ allows a $K^2$-quasiconformal extension $\tilde{\varphi}:\mathbb R^2\to\mathbb R^2$ with
\begin{equation}
\label{estK}
K< \frac{1}{2^{10}} \exp\left\{\big(1+e^{2 \pi}C^5\big)^2\right\}.
\end{equation}

Using the estimate~(\ref{estK}) for the quasiconformal coefficient in Theorem~A, we obtain lower estimates of the first non-trivial eigenvalues in domains satisfy the Ahlfors's 3-point condition.

\vskip 0.2cm
\noindent
{\bf Theorem C.} {\it 
Let a domain $\Omega\subset\mathbb R^2$ is bounded by a Jordan curve $\Gamma$ satisfies the Ahlfors's 3-point condition.
Then
\begin{multline}\label{InHA*}
\frac{1}{\mu_p^{(1)}(\Omega)} \leq 
{} \inf_{q \in [1, 2]} 
\left(\frac{1-\delta}{1/2-\delta}\right)^{(1-\delta)p}  \\
{} \times \frac{2^{p-22} C_\alpha^2 e^{2\left(1+e^{2\pi}C^5\right)^2}}{\pi^{\frac{p}{2}}}
\exp\left\{{\frac{\pi^2(2+ \pi^2)^2e^{2\left(1+e^{2\pi}C^5\right)^2}}{2^{21}\log3}}\right\}\cdot |\Omega|^{\frac{p}{2}}
\end{multline} 
holds for $2<\alpha<\min\left({\frac{2K^2}{K^2-1}}, \gamma*\right)$, where $\delta=1/q-(\alpha-2)/p\alpha$,  $\gamma*$ is the unique solution of the equation 
$$\nu(\alpha):=10^{4 \alpha}\frac{\alpha -2}{\alpha -1}(24\pi^2K^2)^{\alpha}=1$$ and 
$$
C_\alpha=\frac{10^{6}}{[(\alpha -1)(1- \nu(\alpha))]^{1/\alpha}}.
$$
}

\vskip 0.2cm

Using this theorem we obtain lower estimates of $\mu_p^{(1)}$ for snowflakes.

\textbf{Rohde snowflake.} In \cite{Roh} S. Rohde constructed a collection $S$ of snowflake type planar
curves with the intriguing property that each planar quasicircle is bi-Lipschitz equivalent to some curve in $S$.

Rohde's catalog is 
\[
 S:= \bigcup \limits_{1/4 \leq t < 1/2} S_{t}
\] 
where $t$ is a snowflake parameter. Each curve in $S_{t}$ is built in a manner reminiscent of the construction
of the von Koch snowflake. Thus, each $S \in S_{t}$ is the limit of a sequence $S^n$
of polygons where $S^{n+1}$ is obtained from $S^n$ by using the replacement rule illustrated in Figure 3: 
for each of the $4^n$ edges $E$ of $S^n$ we have two choices, either we replace $E$ with the four line segments obtained by dividing
$E$ into four $\text{arcs}$ of equal diameter, or we replace $E$ by a similarity copy of
the polygonal $\text{arc}\, A_{t}$ pictured at the top right of Figure 3. In both cases $E$
is replaced by four new segments, each of these with diameter $(1/4) \text{diam}(E)$
in the first case or with diameter $t\, \text{diam}(E)$ in the second case. The second
type of replacement is done so that the "tip" of the replacement arc points
into the exterior of $S^n$. This iterative process starts with $S^1$ being the unit square, and the snowflake parameter,
thus the polygon $\text{arc}\, A_{t}$, is fixed throughout the construction.

\begin{figure}[h!]
\centering
\includegraphics[width=0.7\textwidth]{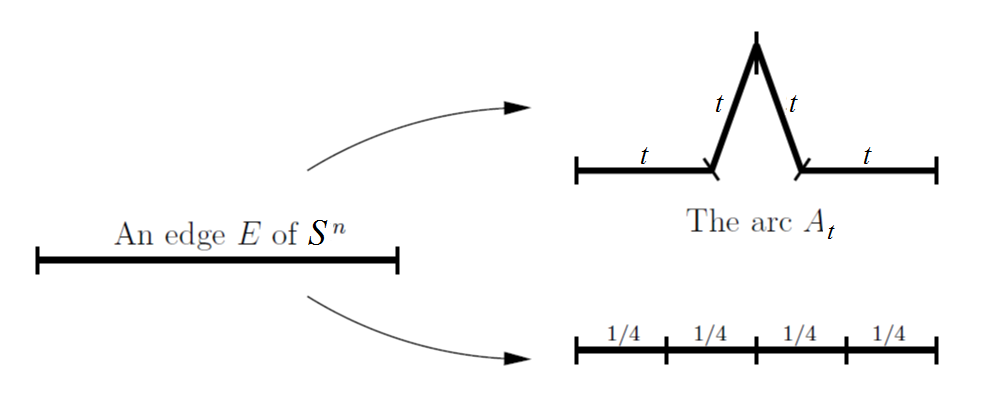}
\caption{Construction of a Rohde-snowflake.}
\end{figure}

The sequence $S^n$ of polygons converges, in the Hausdorff metric, to
a planar quasicircle $S$ that we call a \textit{Rohde snowflake} constructed with
snowflake parameter $t$. Then $S_{t}$ is the collection of all Rohde snowflakes
that can be constructed with snowflake parameter $t$.

In \cite{HM12} established that each Rohde snowflake $S$ in $S_{t}$ is $C$-bounded turning with
\[
C=C(t)=\frac{16}{1-2t}, \quad 1/4 \leq t < 1/2.
\]

A planar curve $\Gamma$ satisfies the $C$-bounded turning, $C \geq 1$, if for each pair of points $x$, $y$, on $\Gamma$,
the smaller diameter $\text{subarc}\, \Gamma[x,y]$ of $\Gamma$ that joins $x$, $y$ satisfies
\begin{equation}\label{BT-cond}
\text{diam}(\Gamma[x,y]) \leq C|x-y|.
\end{equation} 

\begin{figure}[h!]
\centering
\includegraphics[width=0.5\textwidth]{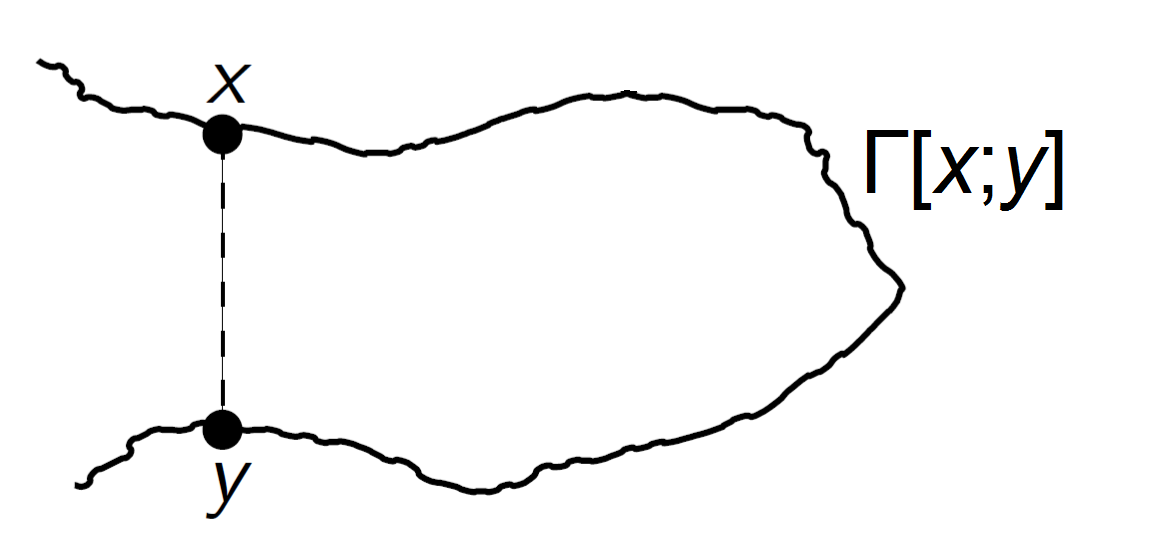}
\caption{$C$-bounded turning condition.}
\end{figure}  

The $C$-bounded turning condition~\eqref{BT-cond} is equivalent the Ahlfors's 3-point condition~\eqref{A-cond} with the same constant $C$ \cite{GR}.

According to Theorem~C and by a known fact that
any $L$-bi-Lipschitz planar homeomorphism is K-quasiconformal with $K=L^2$ we obtain
the following lower estimates of the first non-trivial eigenvalue of the degenerate 
$p$-Laplace Neumann operator in domains type a Rohde snowflakes: 

\vskip 0.2cm
{\it 
Let $S_{t}\subset\mathbb R^2$, $1/4 \leq t < 1/2$, be the Rohde snowflake. Then the following inequality
\begin{align*}
\frac{1}{\mu_p^{(1)}(S_{t})} &\leq 
{} \inf_{q \in [1, 2]} 
\left(\frac{1-\delta}{1/2-\delta}\right)^{(1-\delta)p}  \\
{} &\times \frac{2^{p-22} C_\alpha^2 e^{4\left(1+e^{4\pi}(16/(1-2t))^5\right)^2}}{\pi^{\frac{p}{2}}} \\
{} &\times \exp\left\{{\frac{\pi^2(2+ \pi^2)^2e^{4\left(1+e^{4\pi}(16/(1-2t))^5\right)^2}}{2^{21}\log3}}\right\}\cdot |S_{t}|^{\frac{p}{2}}
\end{align*} 
holds for $2<\alpha<\frac{2K^2}{K^2-1}$, where $\delta=1/q-(\alpha-2)/p\alpha$,
$$
C_\alpha=\frac{10^{6}}{[(\alpha -1)(1- \nu)]^{1/\alpha}}, \quad \nu = 10^{4 \alpha}\frac{\alpha -2}{\alpha -1}\left(\frac{3\pi^2}{2^{17}}e^{4\left(1+e^{2\pi}(16/(1-2t))^5\right)^2}\right)^{\alpha}<1.
$$
} 

\vskip 0.2cm


\vskip 0.3cm

Department of Mathematics, Ben-Gurion University of the Negev, P.O.Box 653, Beer Sheva, 8410501, Israel 
 
\emph{E-mail address:} \email{vladimir@math.bgu.ac.il} \\           
       
 Department of Higher Mathematics and Mathematical Physics, Tomsk Polytechnic University, 634050 Tomsk, Lenin Ave. 30, Russia;
 Department of General Mathematics, Tomsk State University, 634050 Tomsk, Lenin Ave. 36, Russia

 \emph{Current address:} Department of Mathematics, Ben-Gurion University of the Negev, P.O.Box 653, 
  Beer Sheva, 8410501, Israel  
							
 \emph{E-mail address:} \email{vpchelintsev@vtomske.ru}   \\
			  
	Department of Mathematics, Ben-Gurion University of the Negev, P.O.Box 653, Beer Sheva, 8410501, Israel 
							
	\emph{E-mail address:} \email{ukhlov@math.bgu.ac.il}

\end{document}